\documentclass[11pt, article, one side]{memoir}

\settrims{0pt}{0pt} 
\settypeblocksize{*}{34pc}{*} 
\setlrmargins{*}{*}{1} 
\setulmarginsandblock{1in}{1in}{*} 
\setheadfoot{\onelineskip}{2\onelineskip} 
\setheaderspaces{*}{1.5\onelineskip}{*} 
\checkandfixthelayout

\usepackage{mathtools}
\usepackage{amsthm}
\usepackage{dutchcal}
\usepackage{newpxtext}
\usepackage[varg,bigdelims]{newpxmath}
\usepackage{tikz}
\usepackage[inline]{enumitem}
\usepackage[bookmarks=true, colorlinks=true, linkcolor=blue!50!black,
citecolor=orange!50!black, urlcolor=orange!50!black, pdfencoding=unicode]{hyperref}
\usepackage[capitalize]{cleveref}
\usepackage[backend=biber, backref=true, maxbibnames = 10, style = alphabetic]{biblatex}
\usepackage{ytableau} 

	\crefalias{chapter}{section}
  \hypersetup{final}

  \setlist{nosep}

  \usetikzlibrary{ 
  	cd,
  	math,
  	decorations.markings,
  	decorations.pathmorphing,
		decorations.pathreplacing,
  	positioning,
  	arrows.meta,
  	shapes,
		shadows,
		shadings,
  	calc,
  	fit,
  	quotes,
  	intersections,
    circuits,
    circuits.ee.IEC
  }
  
	\tikzcdset{arrow style=tikz, diagrams={>=To}}

  \theoremstyle{theorem}
  \newtheorem{theorem}[section]{Theorem}
  \newtheorem{proposition}[section]{Proposition}
  \newtheorem{corollary}[section]{Corollary}
  \newtheorem{lemma}[section]{Lemma}
  
  \theoremstyle{definition}
  \newtheorem{definition}[section]{Definition}

  \newtheorem*{axiom*}{Axiom}
  
  \theoremstyle{remark}
  \newtheorem{example}[section]{Example}
  \newtheorem{remark}[section]{Remark}
  
\newcommand{\adjr}[5][30pt]{
\begin{tikzcd}[ampersand replacement=\&, column sep=#1]
  #2\ar[r, shift left=6pt, "#3"]\&
  #5\ar[l, shift left=6pt, "#4"]
  \ar[l, phantom, "\scriptstyle\top"]
\end{tikzcd}
}

\DeclareSymbolFont{stmry}{U}{stmry}{m}{n}
\DeclareMathSymbol\fatsemi\mathop{stmry}{"23}

\DeclareFontFamily{U}{mathx}{\hyphenchar\font45}
\DeclareFontShape{U}{mathx}{m}{n}{
      <5> <6> <7> <8> <9> <10>
      <10.95> <12> <14.4> <17.28> <20.74> <24.88>
      mathx10
      }{}
\DeclareSymbolFont{mathx}{U}{mathx}{m}{n}
\DeclareFontSubstitution{U}{mathx}{m}{n}
\DeclareMathAccent{\widecheck}{0}{mathx}{"71}

\DeclareMathOperator*{\colim}{colim}

\DeclareMathOperator{\ob}{Ob}

\newcommand{\cat}[1]{\mathcal{#1}}
\newcommand{\Cat}[1]{\mathsf{#1}}

\newcommand{\To}[1]{\xrightarrow{#1}}

\newcommand{\op}{^\tn{op}}

\newcommand{\tn}[1]{\textnormal{#1}}

\newcommand{\LMO}[2][over]{\ifthenelse{\equal{#1}{over}}{\overset{#2}{\bullet}}{\underset{#2}{\bullet}}}

\newcommand{\nn}{\mathbb{N}}

\newcommand{\inv}{^{-1}}

\newcommand{\smset}{\Cat{Set}}
\newcommand{\smcat}{\Cat{Cat}}
\newcommand{\fin}{\Cat{Fin}}
\newcommand{\cont}{\Cat{Cont}}

\newcommand{\bun}{\Cat{Bun}}

\newcommand{\zero}[1]{#1_\text{zc}}

\newcommand{\yon}{\mathcal{y}}
\newcommand{\poly}{\Cat{Poly}}
\newcommand{\dir}{\Cat{Dir}}

\newcommand{\mdot}{{\cdot}}

\newcommand{\qqand}{\qquad\text{and}\qquad}

\setsecnumdepth{subsubsection}
\settocdepth{subsection}
\begin{document}

\title{Dirichlet Polynomials}
\title{Dirichlet Polynomials form a Topos}

\author{David I.\ Spivak%
\and David Jaz Myers
}

\maketitle
\begin{abstract}
One can think of power series or polynomials in one variable, such as $P(\yon)=2\yon^3+\yon+5$, as functors from the category $\smset$ of sets to itself; these are known as polynomial functors. Denote by $\poly_\smset$ the category of polynomial functors on $\smset$ and natural transformations between them. The constants $0,1$ and operations $+,\times$ that occur in $P(\yon)$ are actually the initial and terminal objects and the coproduct and product in $\poly_\smset$. 

Just as the polynomial functors on $\smset$ are the copresheaves that can be written as sums of representables, one can express any Dirichlet series, e.g.\ $\sum_{n=0}^\infty n^\yon$, as a coproduct of representable presheaves. A Dirichlet polynomial is a finite Dirichlet series, that is, a finite sum of representables $n^{\yon}$. We discuss how both polynomial functors and their Dirichlet analogues can be understood in terms of bundles, and go on to prove that the category of Dirichlet polynomials is an elementary topos.

\end{abstract}

\chapter{Introduction}\label{chap.intro}

Polynomials $P(\yon)$ and finite Dirichlet series $D(\yon)$ in one variable $\cat{y}$, with natural number coefficients $a_i\in\nn$, are respectively functions of the form
\begin{equation}\label{eqn.poly_dir_simple}
\begin{aligned}
  P(\yon)&=a_n\yon^n+\cdots+a_2\yon^2+a_1\yon^1+a_0\yon^0,\\
  D(\yon)&=a_n n^\yon+\cdots+a_22^\yon+a_11^\yon+a_00^\yon.
\end{aligned}
\end{equation}
The first thing we should emphasize is that the algebraic expressions in \eqref{eqn.poly_dir_simple} can in fact be regarded as \emph{objects in a category}, in fact two categories: $\poly$ and $\dir$. We will explain the morphisms later, but for now we note that in $\poly$, $\yon^2=\yon\times\yon$ is a product and $2\yon=\yon+\yon$ is a coproduct, and similarly for $\dir$. The operators---in both the polynomial and the Dirichlet case---are not just algebraic, they are category-theoretic. Moreover, these categories have a rich structure. 

The category $\poly$ is well studied (see \cite{kock2012polynomial}). In particular, the following are equivalent:
\begin{theorem}{\cite{kock2012polynomial}}\label{thm.poly_equiv}
For a functor $P\colon \fin \to \fin$, the following are equivalent:
\begin{enumerate}
    \item $P$ is polynomial.
    \item $P$ is a sum of representables.
    \item $P$ preserves connected limits -- or equivalently, wide pullbacks.
\end{enumerate}
\end{theorem}

In \cref{thm.dir_equiv} we prove an analogous result characterizing Dirichlet polynomials:
\begin{theorem}
For a functor $D\colon\fin\op \to \fin$, the following are equivalent:
\begin{enumerate}
    \item $D$ is a Dirichlet polynomial.
    \item $D$ is a sum of representables.
    \item $D$ sends connected colimits to limits -- or equivalently, $D$ preserves wide pushouts.
\end{enumerate}
\end{theorem}

We will also show that $\dir$ is equivalent to the arrow category of finite sets, 
\[\dir \simeq \fin^{\to},\]
and in particular that $\dir$ is an elementary topos.

If one allows \emph{arbitary} sums of functors represented by finite sets, one gets \emph{analytic} functors in the covariant case---first defined by Joyal in his seminal paper on combinatorial species \cite{joyal1981combinatorial}---and \emph{Dirichlet} functors in the contravariant case---first defined by Baez and Dolan and appearing in Baez's \emph{This Week's Finds} blog \cite{baez300}. Baez and Dolan also drop the traditional negative sign in the exponent (that is, they use $n^s$ where $n^{-s}$ usually appears), but also find a nice way to bring it back by moving to groupoids. Here, we drop the negative sign and work with finite sets to keep things as simple as possible. Similar considerations hold with little extra work for infinite Dirichlet series or power series, and even more generally, by replacing $\fin$ with $\smset$.


\chapter{Polynomial and Dirichlet functors}

Recall that a \emph{co-representable functor} $\fin\to\fin$ is one of the form $\fin(k, -)$ for a finite set
\[k=\{`1\text{'}, `2\text{'},\ldots,`k\text{'}\}.\]
We denote this functor by $\yon^k$ and say it is \emph{represented by} $k\in\fin$. Similarly, a \emph{(contra-) representable functor} $\fin\op\to\fin$ is contravariant functor of the form $\fin(-,k)$; we denote this functor by $k^\yon$. The functors $\yon^-$ and $-^\yon$ are the contravariant and covariant Yoneda embeddings,
\[
  \yon^k \coloneqq \fin(k,-)
  \qqand
  k^\yon\coloneqq\fin(-,k).
\]
For example $\yon^3(2)\cong8$ and $3^\yon(2)\cong9$.

Note that the functor $0^\yon\not\cong 0$ is not the initial object in $\dir$; it is given by
\[
0^\yon(s)=
\begin{cases}
1&\tn{ if } s=0\\
0&\tn{ if } s\geq1.
\end{cases}
\]
The coefficient $a_0$ of $1=\yon^0$ in a polynomial $P$ is called its \emph{constant} term. We refer to the coefficient $\zero{D}\coloneqq a_0$ of $0^\yon$ in a Dirichlet series $D$ as its \emph{zero-content} term. Rather than having no content, the content of the functor $\zero{D}\mdot0^\yon$ becomes significant exactly when it is applied to zero.

\begin{example}
The reader can determine which Dirichlet polynomial $D(\yon)\in\dir$ as in \cref{eqn.poly_dir_simple} has the following terms
\[
\begin{array}{c|ccccccc}
\yon&
\cdots&
5&
4&
3&
2&
1&
0
\\\hline
D(\yon)&
\cdots&
96&
48&
24&
12&
6&
7
\end{array}
\]
Hint: its zero-content term is $\zero{D}=4$.
\end{example}

The set $P(1)$ (resp.\ the set $D(0)$) has particular importance; it is the set of pure-power terms $\yon^k$ in $P$ (resp.\ the pure-exponential terms $k^\yon$ in $D$). For example if $P=\yon^2+4\yon+4$ and $D=2^\yon+4+4\mdot0^\yon$ then $P(1)=D(0)=9$. 

\begin{definition}\label{def.poly_dir_obs}
A \emph{polynomial functor} is a functor $P\colon\fin\to\fin$ that can be expressed as a sum of co-representable functors. Similarly, we define a \emph{Dirichlet functor} to be a functor $D\colon\fin\op\to\fin$ that can be expressed as a sum of representable presheaves (contravariant functors):
\begin{equation}\label{eqn.finite_sum}
  P=\sum_{i=1}^{P(1)}\yon^{p_i}
  \qqand
  D=\sum_{i=1}^{D(0)}(d_i)^\yon.
\end{equation}
That is, $P(X)=\sum_{i=1}^{P(1)}\fin(p_i,X)$ and $D(X)=\sum_{i=1}^{D(0)}\fin(X,d_i)$ as functors applied to $X\in\fin$.
\end{definition}

See \cref{thm.poly_equiv} above for well-known equivalent conditions in $\poly$ and \cref{thm.dir_equiv} below for a Dirichlet analogue.

\chapter{The categories $\poly$ and $\dir$}

For any small category $C$, let $\fin^C$ denote the category whose objects are the functors $C\to\fin$ and whose morphisms are the natural transformations between them.
\begin{definition}\label{def.poly_dir}
The \emph{category of polynomial functors}, denoted $\poly$, is the (skeleton of the) full subcategory of $\fin^\fin$ spanned by sums $P$ of representable functors. The \emph{category of Dirichlet functors}, denoted $\dir$, is the (skeleton of the) full subcategory of $\fin^{(\fin\op)}$ spanned by the sums $D$ of representable presheaves.
\end{definition}

While we will not pursue it here, one can take $\poly_\smset$ to be the full subcategory of functors $\smset\to\smset$ spanned by small coproducts of representables, and similarly for $\dir_\smset$.

\begin{lemma}\label{lemma.count_maps}
The set of polynomial maps $P\to Q$ and Dirichlet maps $D\to E$ are given by the following formulas:
\[
  \poly(P,Q)\coloneqq\prod_{i\in P(1)}Q(p_i)
  \qqand
  \dir(D,E)\coloneqq\prod_{i\in D(0)}E(d_i).
\]
\end{lemma}

\begin{example}
Let $P=2\yon^2$, $Q=\yon+1$, and let $D=2\cdot2^\yon$ and $E=1+0^\yon$. Then there are nine ($9$) polynomial morphisms $P\to Q$, zero ($0$) polynomial morphisms $Q\to P$, one ($1$) Dirichlet morphism $D\to E$, and eight ($8$) Dirichlet morphisms $E\to D$.
\end{example}

\begin{remark}\label{rem.products_coproducts}
Sums and products of polynomials in the usual algebraic sense agree exactly with sums and products in the categorical sense: if $P$ and $Q$ are polynomials, i.e.\ objects in $\poly$, then their coproduct is the usual algebraic sum $P+Q$ of polynomials, and similarly their product is the usual algebraic product $PQ$ of polynomials. The same is true for $\dir$: sums and products of Dirichlet polynomials in the usual algebraic sense agree exactly with sums and products in the categorical sense.
\end{remark}

\paragraph{Formal structures.}
We review some formal structures of the categories $\poly$ and $\dir$; all are straightforward to prove. There is an adjoint quadruple and adjoint 5-tuple as follows, labeled by where they send objects $n\in\fin$, $P\in\poly$, $D\in\dir$:
\begin{equation}\label{eqn.adjoints_galore}
\begin{tikzcd}[column sep=60pt]
  \fin
  	\ar[r, shift left=8pt, "n" description]
		\ar[r, shift left=-24pt, "n\yon"']&
  \poly
  	\ar[l, shift right=24pt, "P(0)"']
  	\ar[l, shift right=-8pt, "P(1)" description]
	\ar[l, phantom, "\scriptstyle\top"]
	\ar[l, phantom, shift left=16pt, "\scriptstyle\top"]
	\ar[l, phantom, shift right=16pt, "\scriptstyle\top"]
\end{tikzcd}
\hspace{1in}
\begin{tikzcd}[column sep=60pt]
  \fin
  	\ar[r, shift left=40pt, "n\cdot 0^\yon" description]
		\ar[r, shift left=8pt, "n" description]
		\ar[r, shift left=-24pt, "n^\yon"']&
  \dir
  	\ar[l, shift right=24pt, "D(0)" description]
		\ar[l, shift right=-8pt, "D(1)" description]
	\ar[l, phantom, shift right=16pt, "\scriptstyle\bot"]
	\ar[l, phantom, shift right=0pt, "\scriptstyle\bot"]
	\ar[l, phantom, shift right=-16pt, "\scriptstyle\bot"]
	\ar[l, phantom, shift right=32pt, "\scriptstyle\bot"]
\end{tikzcd}
\end{equation}
All five of the displayed functors out of $\fin$ are fully faithful.

For each $k:\fin$ the functors $P\mapsto P(k)$ and $D\mapsto D(k)$ have left adjoints, namely $n\mapsto n\yon^k$ and $n\mapsto n\mdot k^\yon$ respectively. These are functorial in $k$ and in fact extend to two-variable adjunctions $\fin\times\poly\to\poly$ and $\fin\times\dir\to\dir$. Indeed, for $n\in\fin$ and $P,Q\in\poly$ (respectively $D,E\in\dir$), we have
\begin{gather*}
\poly(nP,Q)\cong\poly(P,Q^n)\cong\fin(n,\poly(P,Q)),\\
\dir(nD,E)\cong\dir(D,E^n)\cong\fin(n,\dir(D,E)),
\end{gather*}
where $nP$ and $nD$ denote $n$-fold coproducts and $P^n$ and $D^n$ denote $n$-fold products.

Consider the unique function $0\to 1$. The natural transformation induced by it, denoted $\pi_D\colon D(1)\to D(0)$, is equivalent to two natural transformations on $\dir$ via the adjunctions in \cref{eqn.adjoints_galore}:
\begin{equation}\label{eqn.obtain_pi}
n\mdot0^\yon\to n,\qquad
D(1)\To{\pi_D} D(0),\qquad
n\to n^\yon.
\end{equation}
The one labeled $\pi_D$ is also $D(0!)$, where $0!\colon 0\to 1$ is the unique function of that type.

The composite of two polynomial functors $\fin\to\fin$ is again polynomial, $(P\circ Q)(n)\coloneqq P(Q(n))$; this gives a nonsymmetric monoidal structure on $\poly$. The monoidal unit is $\yon$.

Day convolution for the cartesian product monoidal structure provides symmetric monoidal structure $\otimes\colon\poly\times\poly\to\poly$, for which the monoidal unit is $\yon$. This monoidal structure---like the Cartesian monoidal structure---distributes over $+$ We can write an explicit formula for $P\otimes Q$, with $P,Q$ as in \cref{eqn.finite_sum}:
\begin{equation}\label{eqn.dir_monoidal_product}
  P\otimes
  Q=
  \sum_{i=1}^{P(1)}\sum_{j=1}^{Q(1)}\yon^{p_iq_j}
\end{equation}
We call this the \emph{Dirichlet product} of polynomials, for reasons we will see in \cref{rem.dir_prod_reason}.

The Dirichlet monoidal structure is closed; that is, for any $A,Q:\poly$ we define
\begin{equation}\label{eqn.poly_mon_closed}
[A,Q]\coloneqq\prod_{i:A(1)}Q\circ(a_i\yon),
\end{equation}
for example $[n\yon,\yon]\cong \yon^n$ and $[\yon^n,\yon]\cong n\yon$. For any polynomial $A$ there is an $(-\otimes A)\dashv[A,-]$ adjunction 
\begin{equation}\label{eqn.hom_tensor}
	\poly(P\otimes A,Q)\cong\poly(P,[A,Q]).
\end{equation}
In particular we recover \cref{lemma.count_maps} using  \cref{eqn.poly_mon_closed,eqn.adjoints_galore}. The cartesian monoidal structure on $\poly$ is also closed, $\poly(P\times A,Q)\cong\poly(P,Q^A)$, and the formula for $Q^A$ is similar to \cref{eqn.poly_mon_closed}:
\[
  Q^A\coloneqq\prod_{i:A(1)}Q\circ (a_i+\yon).
\]


If we define the \emph{global sections} functor $\Gamma\colon \poly\to\fin\op$ by $\Gamma P\coloneqq\poly(P,\yon)$, or explicitly $\Gamma(P)=[P,\yon](1)=\prod_{i}p_i$, we find that it is left adjoint to the Yoneda embedding
\[
	\adjr[50pt]{\fin\op}{n\mapsto\yon^n}{\Gamma P\mapsfrom P}{\poly}.
\]

Each of the categories $\poly$ and $\dir$ has pullbacks, which we denote using ``fiber product notation'' $A\times_CB$. We can use pullbacks in combination with monad units $\eta_P\colon P\to P(1)$ and $\eta_D\colon D\to D(0)$ arising from \cref{eqn.adjoints_galore} to recover \cref{eqn.finite_sum}:
\[
  P=\sum_{i=1}^{P(1)}P\times_{P(1)}`i\text{'}
  \qqand
	D=\sum_{i=1}^{D(0)}D\times_{D(0)}`i\text{'}.
\]

\begin{remark}
 By a result of Rosebrugh and Wood \cite{rosebrugh1994adjoint}, the category of finite sets is characterized amongst locally finite categories by the existence of the five left adjoints to its Yoneda embedding $k \mapsto y^k\colon \fin \to \fin^{\fin\op}$. The adjoint sextuple displayed in \eqref{eqn.adjoints_galore} is just the observation that five of these six functors restrict to the subcategory $\dir$.
\end{remark}

\chapter{$\poly$ and $\dir$ in terms of bundles}\label{chap.poly_dir_bund}
There is a bijection between the respective object-sets of these two categories
\begin{align}
\nonumber
	\ob(\poly)&\To{\cong}\ob(\dir)\\\label{eqn.poly_dir}
	\sum_{i=1}^n\yon^{k_i}&\mapsto\sum_{i=1}^n (k_i)^\yon.
\end{align}
We call this mapping the \emph{Dirichlet transform} and denote it using an over-line $P\mapsto \overline{P}$. We will see in \cref{thm.equivs} that this bijection extends to an equivalence $\poly_{\tn{cart}}\cong\dir_{\tn{cart}}$ between the subcategories of cartesian maps.

\begin{remark}\label{rem.dir_prod_reason}
With the Dirichlet transform in hand, we see why $P\otimes Q$ can be called the Dirichlet product, e.g.\ in \cref{eqn.dir_monoidal_product}. Namely, the Dirichlet transform is strong monoidal with respect to $\otimes$ and the cartesian monoidal structure $\times$ in $\dir$:
\[\overline{P\otimes Q}=\overline{P}\times\overline{Q}.\]

\end{remark}

\begin{proposition}\label{prop.poly_function}
There is a one-to-one correspondence between the set of polynomials in one variable, the set of Dirichlet polynomials, and the set of (isomorphism classes of) functions $\pi\colon s\to t$ between finite sets.
\end{proposition}
\begin{proof}
We already established a bijection $P\mapsto\overline{P}$ between polynomials and finite Dirichlet series in \cref{eqn.poly_dir}.

Given a finite Dirichlet series $D$, we have a function $\pi_D\colon D(1)\to D(0)$ as in \cref{eqn.obtain_pi}. And given a function $\pi\colon s\to t$, define $D_\pi\coloneqq\sum_{i=1}^t(d_i)^\yon$, where $d_i\coloneqq\pi\inv(i)$ for each $1\leq i\leq t$. (N.B.\ Rather than constructing $D_\pi$ from $\pi$ by hand, one could instead use a certain orthogonal factorization system on $\dir$.)

It is easy to see that the roundtrip on Dirichlet series is identity, and that the round-trip for functions is a natural isomorphism.
\end{proof}

We will upgrade \cref{prop.poly_function} to an equivalence $\poly_{\tn{cart}} \simeq \dir_{\tn{cart}}$ between certain subcategories of $\poly$ and $\dir$ in \cref{thm.equivs}.

\begin{example}
Under the identification from \cref{prop.poly_function}, both the polynomial $2\yon^3+\yon^2+3$ and the Dirichlet series $2\mdot3^\yon+1\mdot2^\yon+3\mdot 0^\yon$ correspond to the function
\begin{equation}\label{eqn.bundle}
\begin{tikzpicture}[x=.5cm, y=.35cm, every label/.style={font=\scriptsize}, baseline=(f)]
	\node[label={[above=-5pt]:$1$}] (Ya) {$\bullet$};
	\node[right=1 of Ya,  label={[above=-5pt]:$2$}] (Yc) {$\bullet$};
	\node[right=1 of Yc,  label={[above=-5pt]:$3$}] (Yd) {$\bullet$};
	\node[right=1 of Yd,  label={[above=-5pt]:$4$}] (Ye) {$\bullet$};
	\node[right=1 of Ye,  label={[above=-5pt]:$5$}] (Yf) {$\bullet$};
	\node[right=1 of Yf,  label={[above=-5pt]:$6$}] (Yg) {$\bullet$};
	\node[draw, inner ysep=4pt, fit={($(Ya)+(-1em,3ex)$) (Yg)}] (Y) {};
	\node[left=0 of Y] (Ylab) {$6\cong D(0)\cong$};
  \node[above=4 of Ya, label={[above=-5pt]:$(1,1)$}] (X11) {$\bullet$};
  \node[above=1 of X11, label={[above=-5pt]:$(1,2)$}] (X12) {$\bullet$};
  \node[above=1 of X12, label={[above=-5pt]:$(1,3)$}] (X13) {$\bullet$};
  \node[above=4 of Yc, label={[above=-5pt]:$(2,1)$}] (X31) {$\bullet$};
  \node[above=1 of X31, label={[above=-5pt]:$(2,2)$}] (X32) {$\bullet$};
  \node[above=1 of X32, label={[above=-5pt]:$(2,3)$}] (X33) {$\bullet$};
  \node[above=4 of Yd, label={[above=-5pt]:$(3,1)$}] (X41) {$\bullet$};
  \node[above=1 of X41, label={[above=-5pt]:$(3,2)$}] (X42) {$\bullet$};
  \node [above=4 of Yg] (xend) {};
	\node[draw, inner ysep=3pt, fit={($(X13)+(-1em,3ex)$) ($(xend)+(.4,0)$)}] (X) {};
	\node[left=0 of X] {$8\cong D(1)\cong$};
	\draw[->, shorten <=3pt, shorten >=3pt] (X) to node[left] (f) {$\pi$} (Y);
\end{tikzpicture}
\end{equation}
\end{example}

We can think of a function $\pi\colon s\to t$, e.g.\ that shown in \eqref{eqn.bundle}, as a \emph{bundle} of fibers $\pi\inv(`i\text{'})$, one for each element $`i\text{'}\in t$. In \cref{def.sheaves_bundles} we define two different notions of morphism between bundles. We will see in \cref{thm.equivs} that they correspond to morphisms in the categories $\poly$ and $\dir$.

For any function $\pi'\colon s'\to t'$ and function $f\colon t\to t'$, denote by $f^*(\pi')$ the pullback function as shown
\[
\begin{tikzcd}
	s\times_{t'}s'\ar[r]\ar[d, "f^*(\pi')"']&
	s'\ar[d, "\pi'"]\\
	t\ar[r, "f"']&
	t'\ar[ul, phantom, very near end, "\lrcorner"]
\end{tikzcd}
\]

\begin{definition}\label{def.sheaves_bundles}
Let $\pi\colon s\to t$ and $\pi'\colon s'\to t'$ be functions between finite sets.
\begin{itemize}
	\item a \emph{bundle morphism} consists of a pair $(f,f_\sharp)$ where $f\colon t\to t'$ is a function and $f_\sharp\colon \pi\to f^*(\pi')$ is a morphism in the slice category over $t$;
	\item a \emph{container morphism} consists of a pair $(f,f^\sharp)$ where $f\colon t\to t'$ is a function and $f^\sharp\colon f^*(\pi')\to \pi$ is a morphism in the slice category over $t$.
\end{itemize}

We say a bundle morphism $(f, f_\sharp)$ (resp. a container morphism $(f, f^{\sharp})$) is \emph{cartesian} if $f_{\sharp}$ (resp. $f^{\sharp})$ is an isomorphism.

\begin{figure}
\[
  \begin{tikzcd}
  s\ar[dr, bend right=40pt, "\pi"']&[-5pt]
  t\times_{t'}s'\ar[r]\ar[d, "f^*\pi'"']\ar[from=l, "f_\sharp"]&
  s'\ar[d, "\pi'"]\\&
  t\ar[r, "f"']&
  t'\ar[ul, phantom, very near end, "\lrcorner"]
  \end{tikzcd}
  \hspace{.75in}
  \begin{tikzcd}
  s\ar[dr, bend right=40pt, "\pi"']&[-5pt]
  t\times_{t'}s'\ar[r]\ar[d, "f^*\pi'"']\ar[l, "f^\sharp"']&
  s'\ar[d, "\pi'"]\\&
  t\ar[r, "f"']&
  t'\ar[ul, phantom, very near end, "\lrcorner"]
  \end{tikzcd}
\]
\caption{The categories $\bun$ and $\cont$ have the same objects, functions $\pi\colon s\to t$. Here a morphism $(f,f_\sharp)\colon \pi\to \pi'$ in $\bun$ and a morphism $(f,f^\sharp)\colon \pi\to\pi'$ in $\cont$ are shown.
}
\label{fig.bund_cont_maps}
\end{figure}
Define $\bun$ (resp.\ $\cont$) to be the category for which an object is a function between finite sets and a morphism is a bundle morphism (resp.\ container morphism); see \cref{fig.bund_cont_maps}. Denote by $\bun_{\tn{cart}}$ (resp.\ $\cont_{\tn{cart}}$) the subcategory of cartesian bundle morphisms.
\end{definition}

One may note that $\bun$ is the Grothendieck construction of the self-indexing $\fin_{/(-)} \colon \fin\op \to \smcat$, while $\cont$ is the Grothendieck construction of its point-wise opposite $(\fin_{/(-)})\op \colon \fin\op \to \smcat$.

 The name \emph{container} comes from the work of Abbot, Altenkirch, and Ghani \cites{abbott2003categories}{abbott2005containers}{abbot2003categoriesthesis} (see Remark 2.18 in \cite{kock2012polynomial} for a discussion of the precise relationship between the notion of container and the notion of polynomial and polynomial functor).

\begin{remark}\label{rem.dir_fin2}
By the universal property of pullbacks, $\bun\simeq\fin^{\to}$ is equivalent (in fact isomorphic) to the category of morphisms and commuting squares in $\fin$. Furthermore, $\bun_{\tn{cart}}$ is equivalent to the category of morphisms and pullback squares in $\fin$, and $\bun_{\tn{cart}} \simeq \cont_{\tn{cart}}$ (as in both cases a cartesian morphism $(f, f_{\sharp})$ or $(f, f^{\sharp})$ is determined by $f$ alone).
\end{remark}

%
%
 Next we show that $\bun\simeq\dir$ is also equivalent to the category of Dirichlet functors, from \cref{def.poly_dir}. Recall that a natural transformation is called \emph{cartesian} if its naturality squares are pullbacks.

\begin{theorem}\label{thm.equivs}
We have equivalences of categories
\[
\poly\simeq\cont
\qqand
\dir\simeq\bun.
\]
In particular, this gives an equivalence $\poly_{\tn{cart}} \simeq \dir_{\tn{cart}}$ between the category of polynomial functors and cartesian natural transformations and the category of Dirichlet functors and cartesian natural transformations.
\end{theorem}
\begin{proof}
The functors $P_-\colon\cont\to\poly$ and $D_-\colon\bun\to\dir$ are defined on each object, i.e.\ function $\pi\colon s\to t$, by the formula $\pi\mapsto P_\pi$ and $\pi\mapsto D_\pi\coloneqq\overline{P_\pi}$ as in \cref{prop.poly_function}. For each $1\leq i\leq t$, denote the fiber of $\pi$ over $i$ by $k_i\coloneqq\pi\inv(i)$.

For any finite set $X$, consider the unique map $X!\colon X\to 1$. Applying $P_-$ and $D_-$ to it, we obtain the corresponding representable: $P_{X!}\cong\yon^X$ and $D_{X!}\cong X^\yon$. We next check that there are natural isomorphisms
 \begin{gather}\nonumber
  \poly(P_{X!},P_\pi)\cong 
  P_\pi(X)=
  \sum_{}X^{k_i}\cong
  \cont(X!, \pi),
  \\\label{eqn.dir_bund}
  \dir(D_{X!}, D_\pi)\cong 
  D_\pi(X)=
  \sum_{i=1}^{t}(k_i)^X\cong
  \bun(X!, \pi).
\end{gather}
In both lines, the first isomorphism is the Yoneda lemma and the second is a computation using \cref{def.sheaves_bundles} (see \cref{fig.bund_cont_maps}). Thus we define $P_-$ on morphisms by sending $f\colon\pi\to\pi'$ in $\cont$ to the ``compose-with-$f$'' natural transformation, i.e.\ having $X$-component $\cont(X!,f)\colon\cont(X!,\pi)\to\cont(X!,\pi')$, which is clearly natural in $X$. We define $D_-$ on morphisms similarly: for $f$ in $\bun$, use the natural transformation $\bun(-!,f)$.

By definition, every object in $\poly$ and $\dir$ is a coproduct of representables, so to prove that we have the desired equivalences, one first checks that coproducts in $\cont$ and $\bun$ are taken pointwise:
\[
(\pi\colon s\to t)+(\pi'\colon s'\to t')\cong(\pi+\pi')\colon (s+s')\to (t+t'),
\]
and then that $P_{\pi+\pi'}=P_\pi+P_{\pi'}$ and $D_{\pi+\pi'}=D_\pi+D_{\pi'}$; see \cref{rem.products_coproducts}.

By \cref{rem.dir_fin2}, we know that $\bun_{\tn{cart}} \simeq \cont_{\tn{cart}}$, and we have just established the equivalences $\poly \simeq \cont$ and $\dir \simeq \bun$. It thus remains to check that the latter equivalences identify cartesian natural transformations in $\poly$ with cartesian morphisms in $\cont$, and similarly for $\dir$ and $\bun$. For polynomial functors, we may refer to \cite[Section 2]{kock2012polynomial}.

Turning to Dirichlet functors, we want to show that for any $f\colon D\to D'$ the square
\begin{equation}\label{eqn.cart1}
\begin{tikzcd}
	D(1)\ar[r, "f_1"]\ar[d, "\pi"']&
	D'(1)\ar[d, "\pi'"]\\
	D(0)\ar[r, "f_0"']&
	D'(0)
\end{tikzcd}
\end{equation}
is a pullback in $\smset$ iff for all functions $g\colon X\to X'$, the naturality square
\begin{equation}\label{eqn.cart2}
\begin{tikzcd}
  D(X')\ar[r, "f_{X'}"]\ar[d, "D(g)"']&
  D'(X')\ar[d, "D'(g)"]\\
  D(X)\ar[r, "f_X"']&
  D'(X)
\end{tikzcd}
\end{equation}
is a pullback in $\smset$; we will freely use the natural isomorphism $D_\pi(X)\cong\bun(X!,\pi)$ from \cref{eqn.dir_bund}. The square in \cref{eqn.cart1} is a special case of that in \cref{eqn.cart2}, namely for $g\coloneqq 0!$ the unique function $0\to 1$; this establishes the only-if direction.

To complete the proof, suppose that \cref{eqn.cart1} is a pullback, take an arbitrary $g\colon X\to X'$, and suppose given a commutative solid-arrow diagram as shown:
\[
\begin{tikzcd}[sep=small]
  X\ar[rr, "g"]\ar[dd]\ar[rd]&&
  X'\ar[dr]\ar[dd]\ar[dl, dotted]\\&
  D(1)\ar[rr, crossing over]&&
  D'(1)\ar[dd]\\
  1\ar[dr]\ar[rr, equal]&&
  1\ar[dr]\ar[dl, dotted]\\&
  D(0)\ar[from=uu, crossing over]\ar[rr]&&
  D'(0)
\end{tikzcd}
\]
We can interpret the statement that \cref{eqn.cart2} is a pullback as saying that there are unique dotted arrows making the diagram commute, since $DX \cong \bun(X!,  D0!)$ and similarly for the other corners of the square in \cref{eqn.cart2}. So, we need to show that if the front face is a pullback, then there are unique diagonal dotted arrows as shown, making the diagram commute. This follows quickly from the universal property of the pullback.
\end{proof}

\begin{corollary}\label{cor.dir_topos}
$\dir$ is an elementary topos.
\end{corollary}
\begin{proof}
For any finite category $C$, the functor category $\fin^C$ is an elementary topos. The result now follows from \cref{rem.dir_fin2,thm.equivs}, noting that $\dir \simeq \fin^{\to}$.
\end{proof}

As we mentioned in the introduction, this all goes through smoothly when one drops all finiteness conditions. The general topos of Dirichlet functors is the category of (arbitrary) sums of representables $\smset\op \to \smset$, and this is equivalent to the arrow category $\smset^{\to}$ and so is itself a topos.

We conclude with the equivalence promised in \cref{chap.intro}.

\begin{theorem}\label{thm.dir_equiv}
A functor $D\colon \fin\op \to \fin$ is a Dirichlet polynomial if and only if it preserves connected limits, or equivalently wide pullbacks.
\end{theorem}
\begin{proof}
Let $D(\yon)=\sum_{i:D(0)}(d_i)^\yon$, and suppose that $J$ is any connected category. Then for any diagram $X\colon J\to\fin$, we have
\begin{align*}
    D(\colim X_j) &= \sum_{i:D(0)} (d_i)^{\colim X_j} \\
    &\cong \sum_{i:D(0)} \lim (d_i)^{X_j} \\
    &\cong \lim \sum_{i:D(0)} (d_i)^{X_j} \\
    &= \lim D(X_j)
\end{align*}
since connected limits commute with sums in any topos (in particular $\smset$).

Now suppose $D\colon\fin\op\to\fin$ is any functor that preserves connected limits; in particular, it sends wide pushouts to wide pullbacks. Every finite set $X$ can be expressed as the wide pushout
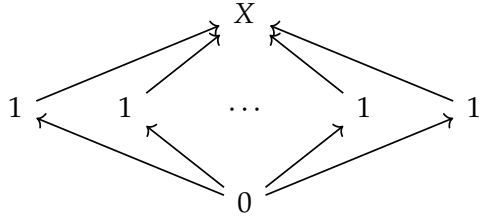
\[
\begin{tikzcd}
              &              & X                                               &              &               \\
1 \arrow[rru] & 1 \arrow[ru] & \cdots                                          & 1 \arrow[lu] & 1 \arrow[llu] \\
              &              & 0 \arrow[llu] \arrow[lu] \arrow[ru] \arrow[rru] &              &              
\end{tikzcd}
\]
of its elements. Therefore, we have the following limit diagram:
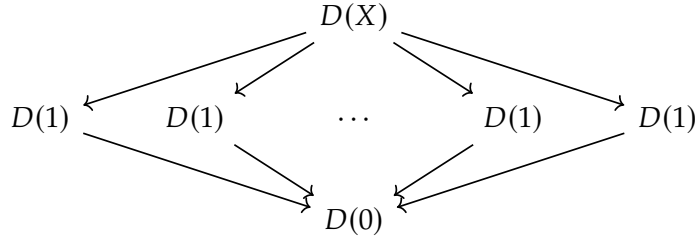
\[
\begin{tikzcd}
                 &                 & D(X) \arrow[lld] \arrow[ld] \arrow[rd] \arrow[rrd] &                 &                  \\
D(1) \arrow[rrd] & D(1) \arrow[rd] & \cdots                                             & D(1) \arrow[ld] & D(1) \arrow[lld] \\
                 &                 & D(0)                                               &                 &                 
\end{tikzcd}
\]
That is, an element of $D(X)$ is a family of elements $a_x \in D(1)$, one for each $x \in X$, such that the $D(0!)(a_x)$ are all equal in $D(0)$. But this is just a bundle map, i.e.
\[D(X) \cong \bun(X!, D(0!))\]
where $X!\colon X\to 1$ and $D(0!)\colon D(1)\to D(0)$. Thus by \cref{thm.equivs}, the functor $D$ is the Dirichlet polynomial associated to the bundle $D(0!)$.
\end{proof}

\section*{Acknowledgments}

The authors thank Joachim Kock, Andr\'{e} Joyal, and Brendan Fong for helpful
comments that improved the quality of this note. Spivak also appreciates support
by Honeywell Inc.\ as well as AFOSR grants FA9550-17-1-0058 and
FA9550-19-1-0113. Jaz Myers appreciates support by his advisor Emily Riehl and the National Science Foundation grant DMS-1652600.

\printbibliography

\end{document}